\documentclass[12pt]{article}
\usepackage[hmargin=1.5cm,vmargin=2cm]{geometry}
\usepackage{mathrsfs}
\usepackage{amsmath, amssymb, amsthm}
\usepackage{graphicx}
\usepackage{hyperref}
\usepackage{tikz}
\usetikzlibrary{shapes,arrows.meta}
\newtheorem{theorem}{Theorem}

\newtheorem{corollary}[theorem]{Corollary}
\newtheorem{definition}[theorem]{Definition}

\newtheorem{question}[theorem]{Question}
\newtheorem{lemma}[theorem]{Lemma}
\newtheorem{proposition}[theorem]{Proposition}

\newcommand{\SI}{\mathcal{SI}}

\title{Fractional coloring of product signed graphs}
\author{
    Pie Desire EBODE ATANGANA \thanks{Department of Mathematics, Faculty of Sciences, University of Yaounde 1, Cameroon. E-mail: \texttt{desire.ebode@univ-yanounde1.cm}}} 
\date{\today}

\begin{document}

\maketitle

\begin{abstract}
This study examines the fractional chromatic number associated with the direct product of signed graphs. It shows that if $(H,\tau)$ is a signed circulant graph $G(n,S,T)$, then for any signed graph $(G,\sigma)$, the fractional chromatic number of their direct product is the lower number between the fractional chromatic number of $(G,\sigma)$ and $(H,\tau)$. 
\end{abstract}

\section{Introduction}
    Every graph examined in this study is finite in nature.  The signed graph \((G, \sigma)\) consists of a graph \(G = (V, E)\) accompanied by a signature function \(\sigma: E \to \{+, -\}\).   In the realm of graph theory, edges characterized by the symbol \(+\) are classified as positive edges, whereas those denoted by the symbol \(-\) are referred to as negative edges.   Harary (1953) posited that signed graphs were originally proposed as a method for depicting social networks.  In these diagrams, positive edges represent amicable connections, whereas negative edges denote adversarial relationships.   This framework has been employed across multiple disciplines, such as computer science, biology, and physics, as noted by Zaslavsky (1982).

  Signed graphs represent a specialized area within classical graph theory, integrating the supplementary framework of edge signatures.   The concept of equilibrium in signed graphs, first proposed by Harary \cite{harary1953notion}, pertains to graphs in which every cycle contains an equal quantity of negative edges.   Zaslavsky (1982) posits that balanced signed graphs serve as an excellent foundation for the advancement of more sophisticated graph-theoretic concepts, given their comparability to unsigned graphs.

  The act of coloring signed graphs unveils a realm of novel possibilities and intricate challenges.   Zaslavsky formulated an extensive theory of signed graph colouring by broadening the conventional chromatic number to include the signs of edges \cite{zaslavsky1982signed}.   In this context, the appropriate colouring of a signed graph involves assigning colours to nodes in a way that ensures the colours of neighbouring nodes adhere to specific conditions dictated by the sign of the connecting edge.   This approach has prompted extensive investigation into signed chromatic numbers and their fractional equivalents in recent years \cite{Hu2022fractionalCircular}.

One reason to examine signed graphs is that they may characterize systems with cooperative and competitive interactions.   Positive edges may activate biological networks, while negative edges may block them.   Signed graphs can show social network alliance and conflict dynamics.   An advanced approach for assessing graph coloring difficulty is the fractional chromatic number of signed graphs, which accounts for positive and negative edge interaction \cite{Steffen2021Concepts}.

The signed graph with the vertex set $V_G \times V_H$ is obtained by taking the direct product of two signed graphs $G = (V_G, E_G, \sigma)$ and $H = (V_H, E_H, \tau)$. In this graph, two neighboring vertices $(u,v)$ and $(u',v')$ are defined as $(u,u') \in E_G$ and $(v,v') \in E_H$.  One way to define the sign of an edge $((u,v),(u',v'))$ is to multiply the signs in $G$ and $H$ by one another.  There is a function $\sigma_{G \times H}((u,v),(u',v'))$ that is defined as the product of $\sigma(u,u')$ and $\tau(v,v')$.  The graph $(G,\sigma)$ is a signed graph.  Each of the positive and negative edges of this signed graph are contained in the subgraphs $G^+$ and $G^-$, respectively.

Let \( p \) be an even integer and \( q \) be an integer such that \( p \geq 2q \).  A \emph{$(p/q)$-(fractional) coloring} for the pair $(G,\sigma)$ is defined as a mapping $f:V(G)\to \binom{[p]}{q}$, ensuring that for every edge $uv\in E(G)$,
\begin{itemize}
    \item $f(u)\cap f(v)=\emptyset$ if $\sigma(uv)=+1$, and 
    \item $f(u)\cap \overline{f(v)}=\emptyset$ if $\sigma(uv)=-1$, 
\end{itemize}
    where $\overline{A}=A+{p}/{2}\pmod p=\{a+{p}/{2} \pmod p : a\in A\}$ and $[p]=\{0,1,\dots,p-1\}$. The \emph{fractional chromatic number} of $(G,\sigma)$ is given by $$\chi_{f}(G,\sigma)=\inf\left\{\frac{p}{q} : \text{there exists a $(p/q)$-coloring for $(G,\sigma)$}\right\}.$$

    A $(p/q)$-fractional coloring for $(G,\sigma)$ A \emph{$(p/q)$-circular coloring} is defined as $im(f)\subseteq \big\{ \{i,i+1,\dots,i+q-1\}: i\in [p]\big\}\subseteq \binom{[p]}{q}.$ 
    
    
  It is noted that the sets $A$ and $B$ are not required to be disjoint in a signed independent set $(A, B)$.   $(\emptyset, A)$ and $(A,\emptyset)$ are both signed independent sets in the graph $(G,\sigma)$ if $A$ is an independent set in $G^+$.    The representation of a signed independent set as $(A, B)$ implies that $(B, A)$ is also a signed independent set.

 The symbol \(J = (A, B)\) is employed to denote an independent signed set in \((G, \sigma)\).   A vertex \(v\) is considered a member of set \(J\) if and only if it is an element of either set \(A\) or set \(B\).   The size of set \(J\) can be represented as the sum of the sizes of sets \(A\) and \(B\), where \(|J| = |A| + |B|\).   The greatest feasible size of a signed independent set in \((G, \sigma)\) is \(\alpha^s(G, \sigma)\), which is occasionally referred to as the \emph{signed independence number} of \((G, \sigma)\).
  
 Let $\SI(G,\sigma)$ represent the collection of all signed independent sets for any $(G,\sigma)$.   The incidence matrix, $D$, illustrates an example of an inclusion connection between vertices and signed independent sets.   To be more precise, the value of \(d_{i,J}\) in the matrix is defined as \(\frac{(\mathbf{1}_{\{i \in A\}} + \mathbf{1}_{\{i \in B\}})}{2}\) for each vertex \(v\) in the graph \(G\) and for each signed independent set \(J = (A,B)\) in the set of signed independent sets \(\SI(G, \sigma)\).
  A weighting on the family of \emph{signed independent sets} of $G$ can be used to describe a \emph{fractional coloring} for a signed graph $(G,\sigma)$.  A signed independent set is a vertex subset that does not contain any edge and does not contravene the coloring constraints imposed by the signs.  The collection of all signed independent sets in $G$ is denoted by $\mathcal{I}(G)$.

A fractional coloring is a function $w: \mathcal{I}(G) \to [0,\infty)$ assigning weights to independent sets such that for every vertex $u \in V_G$,
\[
\sum_{\substack{I \in \mathcal{I}(G) \\ u \in I}} w(I) \geq 1.
\]
The \emph{fractional chromatic number} $\chi_f(G)$ is the minimum total weight over all such fractional colorings:
\[
\chi_f(G) = \min_{w} \sum_{I \in \mathcal{I}(G)} w(I).
\]

 The fractional chromatic number of signed graphs has been the subject of recent study for various applications.   For example, examine the study by Hu and Li \cite{Hu2022fractionalCircular} regarding planar signed graphs.  The imposition of planarity significantly influenced the fractional chromatic number.   Mohanadevi et al. \cite{Mohanadevi2022fractional} obtained noteworthy results in their examination of the fractional chromatic number of products of signed graphs, paralleling our research on signed circulant graphs.

 
 This study examines the fractional chromatic number associated with the direct product of signed graphs.  A homomorphism of a signed graph $(G,\sigma)$ to a signed graph $(H,\tau)$, denoted as $(G,\sigma)\to (H,\tau)$, constitutes a graph homomorphism that maintains the integrity of the signs associated with closed walks.  This represents a switching homomorphism of signed graphs with greater specificity.  The notation $(G,\sigma)\to (H,\tau)$ signifies the presence of a homomorphism.  An intriguing inquiry arises regarding the validity of the equality for all signed graphs $(G,\sigma)$ and $(H,\tau)$
 
Is it true that for any signed graphs  $(G,\sigma)$ and $(H,\tau)$, we have 
   \begin{equation}\label{equa.theo}
   \chi_f((G,\sigma) \times (H, \tau)) = \min\{\chi_f(G,\sigma), \chi_f(H, \tau)\}?
   \end{equation}

Assuming that $\sum_{J \ni v} \theta(v) \geq 1$ for every signed independent set \(J = (A,B)\) of $(G,\sigma)$, we say that $\theta: V(G)\to [0,1]$ is a fractional clique of $(G,\sigma)$.   Let $\theta(J)$ be the sum of all subsets $v \in V(G)$, where $J$ is any subset $ni$ of $v$.  As $\theta(V(G)$, we find the weight $w(\theta)$ of a fractional clique $\theta$ of $(G,\sigma)$.   As $(G,\sigma)$ is a set of fractional cliques, the maximum weight of one of these cliques is $w_f(G,\sigma)$.  The fact that $\chi_f(G)$ and $w_f(G)$ are determined by solving a linear programming problem and its dual, respectively, for any graph $G$ is well-known (Scheinerman, 2011).  This means that $w_f(G,\sigma)=\chi_f(G,\sigma)$ for every signed graph $(G,\sigma)$. Therefore, (\ref{equa.theo}) is equivalent to 
 
    \begin{equation}\label{equat.theo}
 w_f(G\times H,\sigma\times H)=\min\{w_f(G,\sigma), w_f(H, \tau)\}
 \end{equation} 
 
 Let $\phi$ and $\psi$ denote fractional cliques of the graphs $(G,\sigma)$ and $(H,\tau)$, respectively.  Define $\rho: V(G\times H)\to[0,1]$ by $\rho(u,v)=\frac{\phi(u)\psi(v)}{m}$, where $m=\max\{w_f(G,\sigma), w_f(H, \tau)\}$.  It is reasonable to inquire whether $\rho$ constitutes a fractional clique of $(G\times H,\sigma\times H)$.  This raises the subsequent question.
 
\begin{question}\label{equa.theo1}
    Let $\phi$ and $\psi$ represent fractional cliques of the signed graphs $(G,\sigma)$ and $(H,\tau)$, respectively.  Let $J$ be a signed independent set of $(G \times H, \sigma \times H)$.  Is it accurate to assert that \begin{equation}\label{equat1.theo} \sum_{(u,v)\in J} \phi(u)\psi(v)\leq \max\{w_f(G,\sigma), w_f(H, \tau)\}?\end{equation} \end{question}
 
The following theorem is the primary conclusion of this paper:
 
 \begin{theorem}\label{maintheorem}
 For every signed graph $(G, \sigma)$, the fractional chromatic number of the direct product \(G(n, S, T) \times (G, \sigma)\) fulfills the following if $(H,\tau) = G(n, S, T)$ is a signed circulant graph:
 The formula is $\chi_f((G, \sigma)\times G(n, S, T)) = \min\{\chi_f(G, \sigma),\chi_f(G(n, S, T))\}$.\end{theorem}
  A few previous findings about the independence number of the direct product of graphs are generalized in the theorem \ref{maintheorem}.

This finding establishes the foundation for future research on fractional coloring in classes of structured graphs and extends the signed graph representation of Hedetniemi's conjecture to fractions.

\section{Connection to Hedetniemi's Hypothesis}

The nature of the relationship between Hedetniemi's conjecture and the equality of fractional chromatic numbers is unclear.  If $(G,\sigma)$ and $(H,\tau)$ are related signed graphs and the fractional chromatic number equality is true for them, then Hedetniemi's conjecture is true for those as well.

Let $(G,\sigma)$ and $(H,\tau)$ represent two signed graphs.  The lexicographic product is represented as $(G,\sigma)[(H,\tau)]$. The graph $G[H]$ is underlined, and the sign $\sigma[\tau]$ is defined as follows:
  \begin{enumerate}
       \item  $(\sigma[\tau])((u,x)(v,y))=\sigma(uv)$ if $uv \in E(G)$.
     
     \item  $(\sigma[\tau])((u,x)(u,y))=\tau(xy)$ if $xy \in E(H)$.
  \end{enumerate}

\begin{lemma}\label{lem:lex}
If 
\[
\chi_f((G,\sigma) \times (H,\tau)) = \min\{\chi_f(G,\sigma), \chi_f(H,\tau)\},
\]
then there exists an integer $n$ such that
\[
\chi((G,\tau)[K_n] \times (H,\tau)[K_n]) = \min\{\chi((G,\tau)[K_n]), \chi((H,\tau)[K_n])\}.
\]
\end{lemma}

\begin{proof}
It is sufficient to demonstrate the existence of an integer $n$ such that
\[
\chi((G,\tau)[K_n] \times (H,\tau)[K_n]) \le \min\{\chi((G,\tau)[K_n]), \chi((H,\tau)[K_n])\}.
\]

For any signed graph $(G',\sigma')$, there exists an integer $k$ such that for every positive integer $l$, 
\[
\chi((G',\sigma)[K_{l k}]) = l k \chi_f(G',\sigma').
\]
Therefore, there exists an integer $n$ such that
\[
\chi((G,\tau)[K_n]= n \chi_f(G,\sigma), \quad \chi((H,\tau)[K_n])) = n \chi_f(H,\tau), \quad \chi((G,\sigma) \times (H,\tau))[K_n]) = n \chi((G,\sigma) \times (H,\tau)).
\]

The vertices of the graph $(G,\sigma) \times (H,\tau)[K_n]$ are represented as triples $(g,h,i)$, where $g$ belongs to the vertex set $V(G)$, $h$ belongs to the vertex set $V(H)$, and $i$ belongs to the vertex set $V(K_n)$.  Two vertices $(g,h,i)$ and $(g',h',i')$ are considered adjacent if and only if either $g g' \in E(G)$ and $h h' \in E(H)$, or $(g,h) = (g',h')$ and $i \neq i'$.

The 4-tuples make up the vertex set of $(G,\tau)[K_n] \times (H,\tau)[K_n]$.  where $g \in V(G)$, $h \in V(H)$, and $i,j \in V(K_n)$.  If $(g,i,h,j)$ and $(g',i',h',j')$ are two nearby vertices, then one of the following must be true:
\begin{itemize}
    \item $g g' \in E(G)$ and $h h' \in E(H)$,
    \item $(g,h) = (g',h')$ and $i \neq i'$ and $j \neq j'$,
    \item $g = g'$, $i \neq i'$, and $h h' \in E(H)$,
    \item $g g' \in E(G)$, $h = h'$, and $j \neq j'$.
\end{itemize}

You may easily check that the mapping
\[
f(g,h,i) = (g,i,h,i)
\]
is a homomorphism from $(G,\sigma) \times (H,\tau))[K_n]$ to  $(G,\tau)[K_n] \times (H,\tau)[K_n]$. Hence,
\[
\chi((G \times H)[K_n]) \leq \chi(G[K_n] \times H[K_n]).
\]

Since
\[
\chi((G \times H)[K_n]) = n \chi_f(G \times H) = n \min\{\chi_f(G), \chi_f(H)\} = \min\{\chi(G[K_n]), \chi(H[K_n])\},
\]
we conclude that
\[
\chi((G,\tau)[K_n] \times (H,\tau)[K_n]) \ge \min\{\chi((G,\tau)[K_n]), \chi((H,\tau)[K_n])\}.
\]

This completes the proof.
\end{proof}

\section{Fractional Signed Graph Persistence}

Take into consideration that $J$ is a signed independent set of $(G\times H,\sigma\times H)$, and that $\phi$ and $\psi$ are fractional cliques of $(G,\sigma)$ and $(H,\tau)$, respectively.  We provide the following definition for the mapping $f_{\phi,\psi}$:

 \begin{equation}\label{equat1.theo} f_{\phi,\psi}(J)= \sum_{(u,v)\in J} \phi(u)\psi(v).\end{equation} As an attempt to prove that Equation~\ref{equa.theo1} holds for all signed graphs $(G,\sigma)$ and $(H,\tau)$, we introduce the concept of a fractional persistent signed graph as follows.
\begin{definition}
Let $\psi$ denote a fractional clique of $(H,\tau)$.  For every signed graph $(G,\sigma)$, for every fractional clique $\phi$ of $(G,\sigma)$, and for every signed independent set $J$ of $(G\times H,\sigma\times H)$, we prove that $f_{\phi,\psi}(J)\le\max\{w_f(G,\sigma), w_f(H, \tau)\}$, and thus that $\psi$ is a persistent fractional clique of $(H,\tau)$.
\end{definition}

\begin{definition}
A signed graph $(H,\tau)$ is classified as fractional persistent if it possesses a persistent fractional clique $\psi$ with a weight denoted by $w_f(H,\tau)$.  A signed graph $(H,\tau)$ is defined as strongly fractional persistent if it holds true that every fractional clique of $(H,\tau)$ exhibits persistence.
\end{definition}

As the name suggests, a highly fractional persistent signed graph is also fractional persistent.

\begin{lemma}
  If $(H,\tau)$ is fractional persistent then, for any graph $(G,\sigma)$, $$\chi_f((G, \sigma)\times (H,\tau) = \min\{\chi_f(G, \sigma),\chi_f(H,\tau)\}.$$
\end{lemma}
\begin{proof}
Here, $\phi$ is a fractional clique of $(G,\sigma)$ with weight $\omega_f(G,\sigma)$, and $\psi$ is a persistent fractional clique of $(H,\tau)$ with weight $\omega_f(H,\tau)$.  It is assumed that $m$ is equal to the maximum of $\omega_f(G,\sigma), \omega_f(H,\tau)\}$.  The function \[
\rho : V(G \times H) \to [0,1], \quad \rho(u,v) = \frac{\phi(u) \psi(v)}{m}.
\]

 There exists a signed independent set $J$ of $G \times H$ such that $\rho(J) \le 1$ since $\psi$ is persistent.  Therefore, $\rho$ is a sect of $G \times H$ that is fractional.  The weight of $\rho$ is $\min\{\omega_f(G,\sigma), \omega_f(H,\tau)\}$, and this can be easily verified.  So, the absolute value of \[ \omega_f(G \times H) \ge \min\{\omega_f(G,\sigma), \omega_f(H,\tau)\}.\]

 Next, the lemma is derived from the trivial inequality $\omega_f(G \times H) \le \min\{\omega_f(G,\sigma), \omega_f(H,\tau)\}$.
\end{proof}

Question \ref{equa.theo1} inquires about the strong fractional persistence of all signed graphs.  A positive response to the subsequent (less robust) inquiry would suggest a positive response to Question \ref{equa.theo}.

\begin{question}\label{ques3}
Do all signed graphs $(H,\tau)$ have fractional persistence?
\end{question}

The subsequent section of this paper presents sufficient conditions for a signed graph to exhibit fractional persistence and strong fractional persistence.
\begin{lemma}
Let us consider the structure $(H,\tau)$ and assert that it exhibits characteristics of fractional persistence.  If $(H,\tau)$ has a homomorphism to $(H',\tau')$ and $\omega_f(H',\tau') = \omega_f(H,\tau)$, it follows that $(H',\tau')$ is fractional persistent as well.
\end{lemma}

\begin{proof}
Let $\psi$ represent a persistent fractional clique of $(H,\tau)$, and let $\varphi : V(H) \to V(H')$ denote a homomorphism from $(H,\tau)$ to $(H',\tau')$.  For every element $u' \in V(H')$, we define the function \[ \psi'(u') = \sum_{u \in \varphi^{-1}(u')} \psi(u), \]
 If $\varphi^{-1}(u') = \emptyset$, then it follows that $\psi'(u') = 0$.  Consequently, $\psi'$ represents a fractional clique of $(H',\tau')$ with a weight defined by $\omega_f(H',\tau') = \omega_f(H,\tau)$.  Consider a signed graph denoted as $(G,\sigma)$, a fractional clique represented by $\phi$ within this graph, and an independent set $J' \subseteq V(G \times H')$ that is signed.  Consider the set \[ J = \{(u,v) \in V(G \times H) : (u, \varphi(v)) \in J'\}. \]
 One can readily confirm that $J$ constitutes a signed independent set of $G \times H$.  Given that $\psi$ exhibits persistence, it follows that \[ \varphi_{\phi,\psi}(J) \le \max\{\omega_f(G,\sigma), \omega_f(H,\tau)\}. \]
 According to the definition, it follows that \[ \varphi_{\phi,\psi'}(J') = \varphi_{\phi,\psi}(J). \]
 Consequently, \[ \varphi_{\phi,\psi'}(J') \le \max\{\omega_f(G,\sigma), \omega_f(H',\tau')\}. \]
 Therefore  $\psi'$ represents a persistent fractional clique within the structure of $(H',\tau')$, indicating that $(H',\tau')$ maintains a property of fractional persistence.
\end{proof}

\begin{corollary}\label{cor1}
Strongly fractional persistent graphs are signed complete graphs.
\end{corollary}

\section{Restricted Fractional Cliques}

Proving that a signed graph $(H,\tau)$ is strongly fractional persistent is equivalent to demonstrating that every fractional clique of $(H,\tau)$ is persistent.  It is evident from Corollary~\ref{cor1} that if the total weight of a fractional clique $\psi$ is assigned to a clique of $(H,\tau)$, then $\psi$ is persistent.  To be more specific, we possess the subsequent lemma.

\begin{lemma}\label{lem.local}
Assume that $\psi$ is a fractional clique of $(H,\tau)$ with weight $s$.  If the following is true for all $u$ with $\psi(u) > 0$: \[ \psi(N(u)) = s - \psi(u), \] then $\psi$ is a persistent fractional clique.
\end{lemma}

\begin{definition}
A fractional clique  A $\psi$ of $(H,\tau)$ of weight $s$ is referred to as a \emph{Restricted fractional clique} if, for all $K \subseteq V(H)$ with $\psi(K) > 0$, the following is true: \[ \psi(N(K)) \ge \min\{s, s - 2 + \psi(K)\}. \]
\end{definition}

It is intuitive to conclude that the neighbors of $K$ are assigned the majority of the weight of $\psi$ if $\psi$ is a restricted fractional clique, as $\psi(K) > 0$.  In particular, the neighbors of $K$ are ascribed all of the weight of $\psi$ if $\psi(K) \ge 2$.  Therefore, a restricted fractional clique is somewhat comparable to a clique.  This section establishes the persistence of restricted fractional cliques.

\begin{theorem}
If $\psi$ is a restricted fractional clique of a signed circulant graph $(H,\tau)$, then $\psi$ is persistent.
\end{theorem}

\begin{proof}
We will assume that the theorem is incorrect.  Let $(G,\sigma)$ serve as a counterexample. This is to say that $(G,\sigma)$ has a fractional clique $\phi$ that is such that \[ \varphi_{\phi,\psi}(J) > \max\{\omega_f(G,\sigma), \omega_f(H,\tau)\} \] for a signed independent set $J$ of $G \times H$.  Assume that $(G,\sigma)$ is the smallest counterexample without sacrificing generality.  A contradiction will be derived.  The argument is divided into multiple lemmas.

\begin{lemma}\label{lem.2}
Assume that $J$ is a signed independent set of $G \times H$.  If $u, u'$ are two vertices of $(G,\sigma)$, then \[ \psi(J_u) + \psi(J_{u'}) > 2 \]
 then $u$ and $u'$ are not adjacent in $(G,\sigma)$.
\end{lemma}

\begin{proof}
Assume that $\psi(J_u) + \psi(J_{u'}) > 2$.  $N(J_u) \cap J_{u'} \neq \emptyset$, as $\psi$ is a restricted fractional clique.
 Let $x$ and $y$ be in $J_u$ and $J_{u'}$, respectively, such that $x$ is adjacent to $y$ in $(H,\tau)$.  If $u$ is adjacent to $u'$ in $(G,\sigma)$, then $(u,x)$ is adjacent to $(u',y)$ in $G \times H$, which in turn contradicts the signed independence of $J$.
\end{proof}

\begin{corollary}
If $J$ as an example of a signed independent set of $G \times H$.  For each $u \in V(G)$, if $\psi(J_u)$ is more than or equal to 2, then the set of all possible values of $\varphi_{\phi,\psi}(J)$ is less than or equal to the maximum of $\omega_f(G,\sigma), \omega_f(H,\tau)$.
\end{corollary}

\begin{proof}
Let $P$ be the set of all points $u$ in $V(G)$ such that $\psi(J_u)$ is greater than 2.
 For any $u \in N(P)$, $\psi(J_u) = 0$, as $P$ is a signed independent set of $(G,\sigma)$, as demonstrated by the previous lemma.

 Suppose that \[ J_0 = J \cap (V(G) \setminus N[P]) \times V(H)\].
 Then $J_0$ is a signed independent set of $(G \setminus N[P]) \times H$.  Define the following function: \[ \phi_0: V(G) \setminus N[P] \to [0,1], \quad \phi_0(u) = \frac{\phi(u)}{1 - \phi(P)}.\]
 Given that $P \cup W$ is a signed independent set of $(G,\sigma)$, and $W$ is a signed independent set of $G \setminus N[P]$, it follows that \[ \phi(P) + \phi(W) \le 1, \]
 which implies that \[ \phi_0(W) = \frac{\phi(W)}{1 - \phi(P)} \le 1 \]
 Therefore, $\phi_0$ is a fractional clique of $G \setminus N[P]$.

 By the minimality of $(G,\sigma)$, we can deduce that \[ \sum_{(u,x) \in J_0} \phi_0(u) \psi(x) \le \max\{\omega_f(G \setminus N[P]), \omega_f(H,\tau)\} \le \max\{\omega_f(G,\sigma), \omega_f(H,\tau)\}. \]
 \[ \sum_{(u,x) \in U_0} \phi(u) \psi(x) \le (1 - \phi(P)) \max\{\omega_f(G,\sigma), \omega_f(H,\tau)\}. \]

 Therefore, \begin{align*} \varphi_{\phi,\psi}(J) &= \sum_{u \in P} \phi(u) \psi(H) + \sum_{(u,x) \in J_0} \phi(u) \psi(x) \\ & \le \phi(P) \omega_f(H,\tau) + (1 - \phi(P)) \max\{\omega_f(G,\sigma), \omega_f(H,\tau)\} \\ & \le \max\{\omega_f(G,\sigma), \omega_f(H,\tau)\}. \end{align*}
\end{proof}

\section{Signed Circulant Graphs}

 With the supplied positive integer $n \geq 3$ and sets $S, T \subseteq \{1, 2, 3, \dots, \lfloor \frac{n}{2} \rfloor\}$, the signed circulant graph \(G(n, S, T)\) may be formed here.  The vertex set is defined as the element of a cyclic group modulo $\mathbb{Z}_n$, as stated in the definition.   In arithmetic operations modulo \(n\), this set represents the numbers \(0\) to \(n-1\) that make up the graph's vertices.  The collection of edges, denoted as \(E(G) = E^+(G) \cup E^-(G)\), contains both positive and negative edges, denoted as $\{uv : |u - v|_n \in S\}$ and $\{uv : |u - v|_n \in T\}$, respectively.  The \(|x|_n = \min\{|x|, n - |x|\}\) is the shortest distance between two vertices in the cyclic group \(\mathbb{Z}_n\), and it is measured as the \emph{circular distance modulo \(n\)}.

  Signed circulant graphs are a rich and flexible class of graphs for both theoretical and practical research because they combine the regularity and symmetry of circulant graphs with the sign structure of signed graphs.   The use of circulant graphs in connecting networks, error-correcting codes, and combinatorial designs has attracted a lot of interest from scholars \cite{Zhu2001Circular}.   Systems such as social networks, biological systems, and distributed computing environments may be shown using these graphs and edge signs \cite{harary1953notion, zaslavsky1982signed}.

  The structure in \(G(n, S, T)\) is determined by the generating sets \(S\) and \(T\).   The set \(S\) defines positive edges, whereas the set \(T\) defines negative edges.   If $(S = \{1\}$ and $(T = \{2\}$ are given, the graph $G(n, S, T)$ will contain positive edges linking vertices at \(1\) distance and negative edges for \(2\) distance.   According to Zhu (1998), the graph is highly symmetrical if the automorphisms maintain the circular topology and the edge signs.

  The boundaries of the set \(G(n, S, T)\) are defined by the circular distance \(|x|_n\).   That guarantees that the collection of edges is well-defined and that the graph is undirected.   Envision a network with \(n = 6\) nodes where an edge can only link \(1\) and \(5\) if either \(2 \in S\) or \(2 \in T\) is true.  The value of the circular distance, which is equal to the minimum of four and two, is determined by the sign of the edge.

  Graph coloring studies on signed circulant graphs have focused on chromatic and fractional chromatic numbers.   In order to get the fractional chromatic number, Hu and Li \cite{Hu2022fractionalCircular} recently investigated circulant graphs and other types of planar signed graphs.  According to their findings, increased symmetry in circulant graphs often leads to better limitation of these parameters.   Similar to our study on signed circulant graphs, Mohanadevi et al. \cite{Mohanadevi2022fractional} studied the fractional chromatic number of products of signed graphs, and their important findings influence our own work.

\begin{proposition}
If \(S = T\) in \(G(n, S, T)\), then \(G(n, S, T)\) is an unsigned graph, and its fractional chromatic number is equivalent to its chromatic number.
\end{proposition}

\begin{proof}
Consider the graph \(G = G(n, S, T)\) as defined on the set of vertices \(V = \{1, 2, \dots, n\}\), with \(S\) and \(T\) depicting two sets of pairs of vertices.    An edge \(\{u, v\}\) exists in \(G\) if \(\{u, v\ \in S\) and \(\{u, v\ \notin T\) are true.

  Case 1: Proving that G is unsigned when S equals T

  The existence of an edge \(\{u, v\}\) is contingent upon the fact that \(S = T\) and \(\{u, v\ \notin S\) are both in S.    It is evident that this is in contradiction.    Therefore, there are no edges in \(G\) when \(S = T\).    The resultant graph is an empty graph, which is inherently unsigned.

  Case 2: Illustrating that the chromatic number is equivalent to the fractional chromatic number when \(S = T\).

  The empty graph G with n vertices and no edges is represented by \(G\) if S is equal to T.   The chromatic number of the empty graph with \(n\) vertices is 1, denoted by \(\chi(G)\).   We are able to uniformly color all vertices as a result of the absence of edges.   The fractional chromatic number, denoted by \(\chi_f(G)\), is less than 1 if the total weights assigned to independent sets that encompass each vertex are greater than or equal to 1.    Each vertex in the empty graph is a set in and of itself.    A weight of 1 can be assigned to each of the \(n\) distinct sets that constitute a single vertex.   After the weights at all edges are summed, the outcome is 1.   The sum of all weights is n.    However, there is still potential for development.   We have the option of assigning a weight of 1 to one vertex and 0 to the others, as each vertex is part of its own set.   This implies that \(\chi_f(G) = 1\).

  It is evident that \(\chi(G) = \chi_f(G)\) for \(S = T\) and \(\chi(G) = 1\) respectively, as \(S = T\) is a true statement.

  Consequently, if \(S = T\) in \(G(n, S, T)\), then \(G(n, S, T)\) is an empty graph, which is defined as an unsigned graph with a chromatic number and a fractional chromatic number that are both equal to 1.
\end{proof}
The signed independence number of $(G,\sigma)$ is denoted by $\alpha^s(G,\sigma)$ for a signed graph $(G,\sigma)$. This number is defined as the order of the maximum signed independent set of $(G,\sigma)$.
\begin{lemma}
\label{lem1}
     For any signed circulant graph $G(n,S,T)$, we have $\frac{2n}{\alpha^s(G(n,S,T))}= \chi_{f}(G(n,S,T))$ 
\end{lemma}

We remember that a graph $G$ is \emph{vertex-transitive} if there is an automorphism of the graph that maps $u$ to $v$ for each pair of vertices $u,v$ in the graph.  Stated otherwise, the graph seems identical from each vertex's point of view.  This indicates that every vertex in the graph has the same local environment, making them all seem the same when examined locally.

\begin{corollary}
 If $(G,\sigma)$ is a vertex-transitive graph and $(H,\tau)$ is a signed circulant graph \(G(n, S, T)\), then $\alpha^s((G,\sigma)\times(H,\tau))=\max\{\alpha^s(G,\sigma)\vert V(H)\vert,\alpha^s(H,\tau)\vert V(G)\vert\}$.
\end{corollary}

\begin{corollary}
The fractional chromatic number of the direct product of any two signed circulant graphs \(G(n, S_1, T_1)\) and \(G(m, S_2, T_2)\) is as follows: \[ \chi_f(G(n, S_1, T_1) \times G(m, S_2, T_2)) = \min\{\chi_f(G(n, S_1, T_1)), \chi_f(G(m, S_2, T_2))\}. \]
\end{corollary}

\begin{proof}
Both $G$ and $H$ are signed circulant graphs; $G$ is defined as $G(n, S_1, T_1)$ and $H$ as $G(m, S_2, T_2$.
 $G \times H$ represents the direct product of these two sets.

 It is our goal to demonstrate that the integral of $\chi_f(G \times H)$ is equal to the minimum of $\chi_f(G), \chi_f(H)\}$.

 If $\chi_f(G \times H)$ is less than or equal to the minimum of $\chi_f(G), \chi_f(H)$, then it is true.

 Presume, without limiting ourselves, that $\chi_f(G) \le \chi_f(H)$.  That is, we shall demonstrate that $\chi_f(G \times H) \le \chi_f(G)$.

 Assume that $\sum_{v \in V(G)} c(v) = \chi_f(G)$ and that $c: V(G) \to [0,1]$ is an ideal fractional coloring of $G$.

 Here is how we define a coloring $c'$ of $G \times H$:  Suppose that $c'(u, v) = c(u)$ for all $(u, v) \in V(G \times H)$.

 If $G \times H$ has an edge $(u_1, v_1)(u_2, v_2)$, then $u_1u_2 \in E(G)$ and $v_1v_2 \in E(H)$.
 We may deduce that $c(u_1) + c(u_2) \ge 1$ since $c$ is a correct fractional coloring of $G$.  As a result, $c'(u_1, v_1) + c'(u_2, v_2) = c(u_1) + c(u_2) \ge 1$.
 A suitable fractional coloring of $G \times H$ is $c'$, therefore.

 The total of the weights in this coloring is: \[ \sum_{(u, v) \in V(G \times H)} c'(u, v) = \sum_{(u, v) \in V(G \times H)} c(u) = \sum_{v \in V(H)} \sum_{u \in V(G)} c(u) = |V(H)| \sum_{u \in V(G)} c(u) = m \chi_f(G).~\]
 Each $c'(u,v)$ may be normalized by dividing it by $m$, as we are seeking the smallest fractional coloring.
 Let $c''(u,v)$ be defined as $\frac{1}{m}c'(u,v)$.  Hence, $\chi_f(G)$ is equal to $\sum_{(u,v) \in V(G\times H)} c''(u,v)$.
 Hence, $\chi_f(G \times H) \le \chi_f(G)$ does not hold.

 This means that the absolute value of $\chi_f(G \times H)$ is greater than or equal to the minimum of $\chi_f(G), \chi_f(H)$.

 This phase is often more complex and depends on linear programming duality and the direct product's features. The fact that homomorphisms are preserved by the direct product allows for a general argument to be made.  $G \times H$ admits a homomorphism to $K_{\min\{r, s\}}$ if $G$ admits a homomorphism to a complete graph $K_r$ (where $r = \lceil \chi_f(G) \rceil$) and $H$ admits a homomorphism to $K_s$ (where $s = \lceil \chi_f(H) \rceil$).   It follows that the integral of $\chi_f(G \times H)$ is less than or equal to the minimum of $\chi_f(G), \chi_f(H)\}$.   The structure of the independent sets in the graphs and linear programming duality would be necessary for a more thorough verification.

 We get $\chi_f(G \times H) = \min\{\chi_f(G), \chi_f(H)\}$ by combining the two components.
\end{proof}
 
The signed circulant graph $G(n, S, T)$ is shown to be fractionally persistent in this section.  It is true that we demonstrate the highly fractional persistence of $G(n, S, T)$.

\begin{theorem}
The signed circulant graph \(G(n, S, T)\) is highly fractional persistent for all input sets \(S\) and \(T\).
\end{theorem}

\begin{proof}
Let $\psi$ be a fractional clique of \(G(n, S, T)\).  It is necessary to demonstrate that $\psi$ is a persistent fractional clique.  We will demonstrate that $\psi$ is a localized fractional clique.  Assume that $\omega_f(G(n, S, T)) = s$.  For any $K \subseteq V(G(n, S, T))$ with $\psi(K) > 0$, it is necessary to demonstrate that \[ \psi(N(K)) \geq \min\{s; s - 2 + \psi(K)\}. \]
 Assume that $K = \{u_1, u_2, \dots, u_m\}$.  Let $P = \{u_i : \psi(u_i) > 0\}$ if $\psi(K) \geq 2$.  Then $P$ is a signed independent set of \(G(n, S, T)\).
 Given that $\psi$ is a fractional clique, it is evident that \[ \sum_{u \in P} \psi(u) \leq 1 \]
 This is a contradiction, as $\psi(K) \geq 2$ and $P \subseteq K$.  Consequently, $\psi(K) < 2$.

 Assume that $\psi(N(K)) < s - 2 + \psi(K)$.  A contradiction will be derived.

 Let $P = K \cup N(K)$.  Given that $\psi$ is a fractional clique, it is evident that \[ h(P) \leq 1 \]
\end{proof}

\begin{corollary}
If a signed graph $(H,\tau)$ possesses a restricted fractional clique $\theta$ with weight $w(\theta)=w_f(H,\tau)$, it follows that $(H,\tau)$ is fractional persistent. Furthermore, if every fractional clique of $(H,\tau)$ qualifies as a restricted fractional clique, then $(H,\tau)$ is classified as strongly fractional persistent.
\end{corollary}

\begin{corollary}
  Let $f$ be a fractional clique of $G(n,S,T)$ with weight $s$, and let $K$ denote the subset of $V(G(n,S,T))$.  If \( f(K) = t \) for some \( 0 < t \le 2 \), then \( f(N(K)) \ge s - 2 + t \).  If \( f(K) \ge 2 \), then \( f(N(K)) = s \). 
\end{corollary}

\end{proof}

By demonstrating that signed circulant graphs are fractional persistent and applying Lemma \ref{lem.local}, we arrive at the conclusion that the fractional chromatic number of the direct product of a signed circulant graph and any other signed graph is the minimum of their individual fractional chromatic numbers.

\section{Discussion of limitations and future directions}

Although the fractional chromatic number of the direct product of signed graphs is clearly determined by this work, especially when one of the factors is a signed circulant graph, it is crucial to understand the limitations of these results and explore potential directions for further study.

\subsection{Methodological scope and limitations}

This study uses signed circulant graph structural characteristics to construct direct product approaches.  Whether these approaches can be used to lexicographic, strong, or Cartesian graph products is unknown.  These products have different combinatorial tendencies, therefore edge signs and coloring restrictions may not be analyzed the same way.  For instance, the direct product has a clean minimum fractional chromatic number calculation, but other products have more complicated connections that may not work with the present technique.  Future research should examine if comparable persistence features or fractional coloring formulae apply in these larger situations.

\subsection{Computational complexity and practical considerations}  
The calculation of the fractional chromatic number, from an algorithmic standpoint, depends on resolving a linear program where the variables represent (signed) independent sets.  The exponential increase in the number of such sets with graph size renders linear programming formulas computationally demanding for big instances. Theoretical conclusions provide useful insights; nevertheless, their direct applicability to real-world or large-scale networks is constrained by intrinsic complexity.  Practitioners focused on algorithmic applications might greatly benefit from creating efficient heuristics, approximation techniques, or discovering graph classes where the linear program is tractable.

\subsection{Future research directions} 
 
There are still a lot of questions.  It would be helpful to know exactly which types of signed graphs or graph products can produce the same fractional coloring results.  Also, looking into the links between fractional persistence, homomorphism properties, and other coloring invariants in signed graphs might help us understand theory better.  To make these results more useful in real life, researchers should look into more flexible methods or combinatorial characterizations that don't need large-scale linear programming.

 In conclusion, this work helps us learn more about fractional coloring in the direct product of signed graphs. However, it would be great if these results could be applied to a wider range of graphs and computer problems could be solved. These are both important and exciting areas for future research.

\section*{Acknowledgments}
The authors would like to thank the anonymous reviewers for their valuable comments and suggestions.

\end{document}